\documentclass{article}  
\usepackage[%pdftex,
bookmarksnumbered,
bookmarksopen,
colorlinks,
citecolor=blue,
linkcolor=blue,]{hyperref}
\usepackage{amsmath}
\usepackage{amssymb}
\usepackage{mathrsfs}
\usepackage{amsfonts}
\usepackage{amsthm}
\usepackage{epsfig}
\usepackage{graphicx}
\usepackage{subfigure}
\usepackage{epstopdf}
\usepackage{color}
\usepackage{boxedminipage}
\usepackage{cite}

\newtheorem{theorem}{Theorem}[section]

\newtheorem{lemma}{Lemma}[section]

\newtheorem{remark}{Remark}[section]

\newcommand{\T}{\ensuremath{ \mathbb R^{n_1\times \cdots\times n_d}   }}

\newcommand{\bigxiaokuohao}[1]{\ensuremath{ \left(  #1 \right) }}      
\newcommand{\bigjueduizhi}[1]{\ensuremath{ \left|  #1 \right| }}   
\newcommand{\bigdakuohao}[1]{\ensuremath{ \left\{  #1 \right\} }}

\newcommand{\normalxiaokuohao}[1]{\ensuremath{  (  #1  ) }}

\newcommand{\bigfnorm}[1]{\ensuremath{ \left\|   #1 \right\|_F }}    
\newcommand{\bignorm}[1]{\ensuremath{ \left\|   #1 \right\|  }}        
\newcommand{\normalnorm}[1]{\ensuremath{  \|   #1  \|  }}     
\newcommand{\bigzeronorm}[1]{\ensuremath{ \left\|   #1 \right\|_0  }}        
\newcommand{\bigonenorm}[1]{\ensuremath{ \left\|   #1 \right\|_1  }}

\newcommand{\innerprod}[2]{\ensuremath{ \left\langle   #1 , #2\right\rangle }}

\title{Practical Approximation Algorithms for $\ell_1$-Regularized Sparse   Rank-$1$ Approximation to Higher-Order Tensors}

\author{Xianpeng Mao\thanks{School of Physical Science and Technology, Guangxi University, Nanning, 530004, China.} \and Yuning Yang\thanks{College of Mathematics and Information Science, Guangxi University, Nanning, 530004, China.} \thanks{Corresponding author. Email:  yyang@gxu.edu.cn.}}

\begin{document} %\large
\maketitle

%\tableofcontents
%\newpage

\begin{abstract}

	Two approximation algorithms are proposed for $\ell_1$-regularized sparse  rank-1 approximation to higher-order tensors. The algorithms are based on multilinear relaxation and sparsification, which are easily implemented and well scalable.  
	In particular, the second one scales linearly with the size of the input tensor. Based on a careful estimation of the $\ell_1$-regularized sparsification, theoretical approximation lower bounds are derived. Our theoretical results also suggest an explicit way of choosing the regularization parameters. Numerical examples are provided to verify the proposed algorithms. 
	
	%\noindent {\bf  AMS subject classifications.} 90C26, 15A18, 15A69, 41A50
	\hspace{2mm}\vspace{3mm}
	
\end{abstract}

\thispagestyle{plain} \markboth{X. Mao \and Y. Yang}{Approximation Algorithms for Sparse Tensor BR1Approx}

\section{Introduction}

The goal of sparse rank-$1$ approximation to higher-order tensors is to find a sparse rank-$1$ tensor that is as close as possible to the given data tensor in the sense of Euclidean distance. It can be seen as a sparse generalization of the tensor best rank-$1$ approximation \cite{de2000on,qi2018tensor} and a higher-order extension of the  matrix sparse SVD \cite{witten2009a}. 
sparse rank-$1$ approximation to  tensors is important in higher-order PCA \cite{allen2012sparse},  co-clustering \cite{papalexakis2012k}, sparse tensor decomposition \cite{madrid2017tensor,sun2017provable},    and sparse tensor regression  \cite{sun2017store}; just to name a few.

To encourage sparsity for the tensor rank-$1$ approximation problem,  existing literature   employed the $\ell_0$-constraint \cite{sun2017provable,sun2017store} or the $\ell_1$-regularization \cite{allen2012sparse,papalexakis2012k,madrid2017tensor}. %Considering the solution methods, iterative algorithms are the most favorable 
Due to the separable structure of the problem, one usually   considers   block coordinate descent type algorithms to solve the models \cite{allen2012sparse,papalexakis2012k,madrid2017tensor,sun2017provable,sun2017store,wang2020sparse}. However, as the problem is nonconvex, it is hard to measure the solution quality for an iterative algorithm, if we have no a priori information about the data tensor. In the context of sparse matrix  PCA and SVD, approximation algorithms with provable lower bounds have drawn much attention during the past decade; see, e.g., \cite{d2014approximation,aspremont2007a,chan2016approximability}. Approximation algorithms have also been developed for non-sparse rank-$1$ approximation to tensors; see, e.g, \cite{he2010approximation,he2012probability,zhang2012cubic}. Very recently, approximation algorithms were designed for $\ell_0$-constrained sparse rank-$1$ approximation to tensors \cite{mao2022several}. The interest in approximation algorithms is not only on themselves, but also that they can generalize high-quality feasible solutions as initializers for iterative algorithms.

In view of the above, however,   in the $\ell_1$-regularized setting, there  still lacks  approximation algorithms in the literature.   This paper attempts to study this point. In the matrix setting, existing work on sparse matrix  PCA usually prefers convex relaxations \cite{d2014approximation,aspremont2007a}; however, they do not scale well; recent work \cite{bertsimas2022solving} formulated matrix sparse PCA as a mixed-integer semidefinite program that can deal with large-scale problems, but this approach cannot be generalized to our setting due to the multilinearity of the considered problem and the $\ell_1$-regularization term. In contrast to these approaches, in this paper, we 
extend the idea of multilinear relaxation  \cite{he2010approximation,mao2022several} to our setting and combine it with the $\ell_1$-regularization induced subproblems. Although the multilinear relaxation is nonconvex, it indeed enables us to design   easily implemented and well scalable algorithms, where the second one, in particular, scales linearly with the size of the input tensor. For a $d$-th order tensor, $\frac{ \prod^d_{j=1}\bigxiaokuohao{1-\omega_j\sqrt{n_j} +\omega_j }   }{ \sqrt{ \prod^{d-1}_{j=2} n_j } }$-- and $\frac{ \prod^d_{j=1}\bigxiaokuohao{1-\omega_j\sqrt{n_j} +\omega_j }   }{ \sqrt{ \prod^{d-1}_{j=1} n_j } }$--lower bounds are established when $\omega_j < 1/\sqrt{n_j}$, where $\omega_j$'s are the regularization parameters. Although extending the multilinear relaxation to the $\ell_1$-regularized setting seems to be straightforward, the sparsification needs to be specifically designed and
the analysis is somewhat more involved than that of \cite{he2010approximation} for the non-sparse case and \cite{mao2022several} for the $\ell_0$-constrained case. This is because we  have to incorporate the $\ell_1$-regularization term into the lower bound analysis, which requires more careful estimation. Besides, our theoretical results also suggest  a direct way of choosing the regularization parameters, which is  preliminarily confirmed by numerical observations.

The remainder is organized as follows. Sect. \ref{sec:model} introduces the sparse rank-$1$ approximation models while approximation algorithms and approximation bounds are presented in Sect. \ref{sec:approx_alg}. Numerical examples are provided in Sect. \ref{sec:numer_experiments}. Sect. \ref{sec:conclusions} draws some conclusions. 

{\bf Notation.} vectors are written as  $(\mathbf x,\mathbf y,\ldots)$, matrices
correspond to  $(A,B,\ldots)$, and tensors are
written as $(\mathcal{A}, \mathcal{B},
\cdots)$. $\mathbb R^{n_1\times \cdots\times n_d}$ denotes the space of $n_1\times\cdots\times n_d$ real tensors. 
The inner product $\langle \mathcal A,\mathcal B\rangle$ between    two tensors $\mathcal A,\mathcal B$ of the same size is
the sum of entry-wise product. 
The Frobenius  norm of $\mathcal A$ is  $\|\mathcal A\|_F = \langle\mathcal A,\mathcal A\rangle^{1/2}$. %the spectral norm of $\bignorm{\mathcal A}_2$ is the leading singular value of $\mathcal A$.
$\circ$ denotes the outer product and for $\mathbf x_j\in\mathbb R^{n_j}$, $j=1,\ldots,d$, $\mathbf x_1\circ\cdots\circ\mathbf x_d$ denotes a rank-1 tensor in $\T$.  For a vector $\mathbf x$, $\|\mathbf x\|$ is the usual Euclidien norm,   $\|\mathbf x\|_1$ is the $\ell_1$ norm, and $\|\mathbf x\|_0$ is the $\ell_0$ semi-norm; $\bigjueduizhi{\mathbf x}$ means that every entry of $\mathbf x$ takes its absolute value.

\section{$\ell_1$-Regularized Sparse Tensor  Rank-$1$ Approximation} \label{sec:model}

Given  $\mathcal A \in\T$ with $d\geq 3$, the $\ell_0$-constrained sparse tensor rank-$1$ approximation consists of finding a set of sparse vectors $\mathbf x_1,\ldots,\mathbf x_d$
\cite{sun2017provable,sun2017store,mao2022several}:
\begin{equation}
	\label{prob:str1approx_org}
	\begin{split}
		&\min~\bigfnorm{\lambda\cdot\mathbf x_1\circ\cdots\circ \mathbf x_d- \mathcal A}^2\\
		&~~{\rm s.t.}~ \lambda\in\mathbb R, ~\mathbf x_j\in\mathbb R^{n_j},~\bignorm{\mathbf x_j}=1,~\bigzeronorm{\mathbf x_j}\leq r_j,~ 1\leq j\leq d,
	\end{split}
\end{equation}
where $r_j$'s are parameters. $\lambda$ above can be eliminated and \eqref{prob:str1approx_org} is  equivalent to the following maximization problem \cite{mao2022several}:
\begin{equation}
	\label{prob:str1approx_org_max_l0}
	\begin{split}
		&\max  ~ \innerprod{\mathcal A}{\mathbf x_1\circ\cdots\circ\mathbf x_d}\\
		&~~{\rm s.t.}~ \bignorm{\mathbf x_j}=1,~\bigzeronorm{\mathbf x_j}\leq r_j, 1\leq j\leq d.
	\end{split}
\end{equation}

In another thread, the   $\ell_1$-regularized version of the problem takes the following form \cite{allen2012sparse} $(d=3)$:
\begin{equation}
	\label{prob:str1approx_org_max_l1_regularized_inequality_constraint}
	\begin{split}
		&\max ~ \innerprod{\mathcal A}{\mathbf x_1\circ\cdots\circ\mathbf x_d} - \sum^d_{j=1}\nolimits \omega_j\bigonenorm{\mathbf x_j}\\
		&~~{\rm s.t.}~~ \mathbf x_j\in\mathbb R^{n_j},~\bignorm{\mathbf x_j}\leq 1,~1\leq j\leq d,
	\end{split}
\end{equation}
where $\omega_j\geq 0$ denotes the regularization parameter. $\bigonenorm{\mathbf x_j}$ can be replaced by $\bigonenorm{D\mathbf x_j}$ where $D$ is a structure matrix \cite{madrid2017tensor}.  \cite{wang2020sparse} considered solution methods for the following two models $(d=3)$:
\begin{equation}
	\label{prob:str1approx_org_l1_regularized_inequality_constraint}
	\min_{\mathbf x_j\in\mathbb R^{n_j},1\leq j\leq d}\nolimits~ \frac{1}{2}\bigfnorm{\mathbf x_1\circ\cdots\circ \mathbf x_d- \mathcal A}^2 + \sum^d_{j=1}\nolimits \omega_j\bigonenorm{\mathbf x_j};
\end{equation}
\begin{equation}
	\label{prob:str1approx_org_l1_regularized_equality_constraint}
	\begin{split}
		&\min ~  \frac{1}{2}\bigfnorm{\lambda\cdot\mathbf x_1\circ\cdots\circ \mathbf x_d- \mathcal A}^2 + \sum^d_{j=1}\nolimits \omega_j\bigonenorm{\mathbf x_j}\\
		&~~{\rm s.t.}~  \lambda\in\mathbb R,~\mathbf x_j\in\mathbb R^{n_j},~\bignorm{\mathbf x_j}= 1,1\leq j\leq d.
	\end{split}
\end{equation}

In this paper,  we focus on the $\ell_1$-regularized problem of the following form:
\begin{equation}
	\label{prob:str1approx_org_max_l1_regularized}
	\begin{split}
		&\max ~ \innerprod{\mathcal A}{\mathbf x_1\circ\cdots\circ\mathbf x_d} - \sum^d_{j=1}\nolimits \omega_j\bigonenorm{\mathbf x_j}\\
		&~~{\rm s.t.}~~ \mathbf x_j\in\mathbb R^{n_j},~\bignorm{\mathbf x_j}= 1, ~1\leq j\leq d,
	\end{split}
\end{equation}
which is a slight modification of \eqref{prob:str1approx_org_max_l1_regularized_inequality_constraint} with   the ball constraints replaced   by spherical ones. This is for better designing and analyzing the approximation algorithms, and it can avoid zero solutions.

Note that \eqref{prob:str1approx_org_max_l1_regularized} has close relations with the aforementioned problems. Firstly, it is a $\ell_1$-regularized version of \eqref{prob:str1approx_org_max_l0}. Secondly, if $\omega_j$'s are chosen small enough, then the ball constrained model \eqref{prob:str1approx_org_max_l1_regularized_inequality_constraint} boils down to \eqref{prob:str1approx_org_max_l1_regularized}, as the maximizer always lies on the sphere in such a case. Thirdly, considering the spherical constraints, if one minimizes \eqref{prob:str1approx_org_l1_regularized_equality_constraint} with respect to $\lambda$, one obtains $\lambda = \innerprod{\mathcal A}{\mathbf x_1\circ\cdots\circ\mathbf x_d}$. Substituting this into \eqref{prob:str1approx_org_l1_regularized_equality_constraint} gives the objective function $1/2\innerprod{\mathcal A}{\mathbf x_1\circ\cdots\circ\mathbf x_d}^2 - \sum^d_{j=1}\omega_j\bigonenorm{\mathbf x_j}$ (if the $\min$ is replaced by $\max$), which plays a similar role as that in \eqref{prob:str1approx_org_max_l1_regularized}.

\section{Approximation Algorithms and   Approximation Bounds}\label{sec:approx_alg}

We need some preparations before presenting the algorithms. 
For any $\mathbf a\in\mathbb R^n$ and given that $\omega\geq 0$, it is known that the solution to the following problem
\[
\min_{\mathbf x\in\mathbb R^n} \frac{1}{2}\bignorm{\mathbf x-\mathbf a}^2 + \omega \bigonenorm{\mathbf x}
\]
is given by
\begin{equation*}
	\label{eq:l1_prox_tmp}
	S(\mathbf a,\omega) := {\rm sgn}(\mathbf a)\circledast \bigxiaokuohao{ |\mathbf a|- \omega\mathbf e }_+ ,
\end{equation*}
where ${\rm sgn}(\cdot)$ is the sign function,   $\circledast$ denotes the entry-wise product,  $\mathbf e$ represents the all-one vector, and $(\cdot)_+ := \max\bigdakuohao{\cdot,\mathbf 0}$.   The following result holds. 

\begin{lemma} \label{lem:l1_prox_constrained}
	Given a nonzero vector $\mathbf a\in\mathbb R^n$ and $\omega\geq 0$,  it holds that
	\begin{equation} \label{prob:sparsify_problem}
		\begin{split}
			&	~~~~\max_{\mathbf x\in\mathbb R^n,\bignorm{\mathbf x}=1} \nolimits \innerprod{\mathbf a}{\mathbf x} - \omega \bigonenorm{\mathbf x}\\
			&= 	\left\{ \begin{array}{lr}
				\bignorm{S(\mathbf a,\omega)}
				= 	\sqrt{ \sum^n_{i=1}\bigxiaokuohao{|\mathbf{a}(i)| -\omega}_+^2  } , & \bignorm{S(\mathbf a,\omega)}\neq 0,\\
				\max_{1\leq i\leq n}|\mathbf{a}(i)|-\omega	,&\bignorm{S(\mathbf a,\omega)}=0,
			\end{array} \right. 
		\end{split}
	\end{equation}
	which is solved by
	\begin{equation}
		\label{sol:subproblem}
		\mathbf x^* = N(\mathbf a,\omega) := 	\left\{ \begin{array}{lr}
			\frac{S(\mathbf a,\omega)}{\bignorm{S(\mathbf a,\omega)}}, & \bignorm{S(\mathbf a,\omega)}\neq 0,\\
			{\rm sgn}( \mathbf{a}(\hat{i})) \cdot\mathbf e_{\hat i}	,&\bignorm{S(\mathbf a,\omega)}=0,
		\end{array} \right. 
	\end{equation}
	where $\mathbf{a}(i)$ means the $i$-th entry of $\mathbf a$, $\hat i$ is any index such that $|\mathbf{a}(\hat i)|=  \max_{1\leq i\leq n}{ |\mathbf{a}(i)| }$,   $\mathbf e_{\hat i}$ denotes the $\hat i$-th vector of the standard basis of $\mathbb R^n$, and we set ${\rm sgn}( \mathbf{a}(\hat i)) = 1$ if $\mathbf{a}(\hat i)=0$.
\end{lemma}
\begin{proof}For any maximizer $\mathbf x^*$  to the problem $	\max_{\mathbf x\in\mathbb R^n,\bignorm{\mathbf x}=1} \nolimits \innerprod{\mathbf a}{\mathbf x} - \omega \bigonenorm{\mathbf x}$,  we first have the observation that ${\rm sgn}( \mathbf{x}^*(i)) = {\rm sgn}( \mathbf{a}(i))$, otherwise the sign of $ \mathbf{x}^*(i)$ can be reversed to obtain a larger objective value. This means that $\innerprod{\mathbf a}{\mathbf x^*} = \innerprod{\bigjueduizhi{\mathbf a}}{\bigjueduizhi{\mathbf x^*}}$, and so 
	$$	\max_{\mathbf x\in\mathbb R^n,\bignorm{\mathbf x}=1} \nolimits \innerprod{\mathbf a}{\mathbf x} - \omega \bigonenorm{\mathbf x} = 	\max_{\mathbf x\in\mathbb R^n,\bignorm{\mathbf x}=1} \nolimits \innerprod{|\mathbf a|-\omega\mathbf e}{|\mathbf x|}. $$
	Whenever $S(\mathbf a,\omega)\neq \mathbf 0$, this means that $| \mathbf{a}(i)|> \omega$ for at least one $i$, and so 
	$
	\max_{\mathbf x\in\mathbb R^n,\bignorm{\mathbf x}=1} \nolimits \innerprod{\mathbf a}{\mathbf x} - \omega \bigonenorm{\mathbf x} > 0,
	$
	leading to that
	\[
	\max_{\mathbf x\in\mathbb R^n,\bignorm{\mathbf x}=1} \nolimits \innerprod{\mathbf a}{\mathbf x} - \omega \bigonenorm{\mathbf x}  = 	\max_{\mathbf x\in\mathbb R^n,\bignorm{\mathbf x}\leq 1} \nolimits \innerprod{\mathbf a}{\mathbf x} - \omega \bigonenorm{\mathbf x}, 
	\]
	whose maximizer is given by $S( {\mathbf a},\omega)/\bignorm{S( {\mathbf a},\omega)}$ according to 
	\cite[Proposition 4.6]{luss2013conditional}. %that $\mathbf x^* = S(\mathbf a,\boldsymbol{\sigma})/ ||S(\mathbf a,\boldsymbol{\sigma})||$. %Note that \cite{luss2013conditional} requires that $ {\sigma}_i>0$ for all $i$; however, it is not hard to verify that if $ {\sigma}_i\geq 0$ the result still holds. 
	
	%  It follows from  $\innerprod{\mathbf a}{\mathbf x^*}=\innerprod{\bigjueduizhi{\mathbf a}}{\bigjueduizhi{\mathbf x^*}} $ that $	\max_{\mathbf x\in\mathbb R^n,\bignorm{\mathbf x}=1} \nolimits \innerprod{\mathbf a}{\mathbf x} - \omega \bigonenorm{\mathbf x} = 	\max_{\mathbf x\in\mathbb R^n,\bignorm{\mathbf x}=1} \nolimits \innerprod{|\mathbf a|-\omega\mathbf e}{|\mathbf x|} $.
	When $S(\mathbf a,\omega) = \mathbf 0$,  we get $|\mathbf a| - \omega\mathbf e\leq \mathbf 0$, implying that the maximizer is given by $|\mathbf x^*| =\mathbf e_{\hat i}$. This together with $\innerprod{\mathbf a}{\mathbf x^*} = \innerprod{|\mathbf a|}{|\mathbf x^*|}$ shows that $\mathbf x^* = {\rm sgn}(  \mathbf{a}(\hat i)) \mathbf e_{\hat i}$. Finally, the maximum of the problem can be easily derived.
\end{proof}

With the above result at hand,
we begin with the general idea on how to design the algorithms. We   use multilinear relaxation to obtain a set of non-sparse candidate vectors $\mathbf x_1,\ldots,\mathbf x_d$ from  matrix unfoldings and tensor-vector multiplications sequentially,  using certain extraction strategies. To sparsify these vectors, in view of the $\ell_1$-regularization, we solve \eqref{prob:sparsify_problem} with $\mathbf a = \mathbf x_i$ to get the   final solutions. That is to say, the sparse solutions are given by $N(\mathbf x_i,\omega)$. The designed algorithms are presented in Algorithms \ref{proc:init3_order_d} and \ref{proc:init4_order_d}, where the former employs SVD to extract the non-sparse vectors, while the latter uses certain maximal energy rule. In the algorithms, \texttt{reshape}  is the same as that in Matlab. 

%We begin with the illustration from third-order tensors. We will employ the Matlab function \texttt{reshape} to denote tensor folding/unfolding operations. For instance, given $\mathcal A\in\T$, $\mathbf a = \texttt{reshape}(\mathcal A, \prod^d_{j=1}n_j,1)$ means the unfolding of $\mathcal A$ to a vector $\mathbf a$ in $\mathbb R^{\prod^d_{j=1}n_j}$, while $\mathcal A = \texttt{reshape}(\mathbf a,n_1,\ldots, n_d)$ means the folding of   $\mathbf a$ back to   $\mathcal A$.  

\begin{boxedminipage}{0.92\textwidth}\small
	\begin{equation}  \label{proc:init3_order_d}
		\noindent {\rm Algorithm}~ (\mathbf x^{\omega_1}_1,\ldots,\mathbf x^{\omega_d}_d) =  {\rm approx\_alg\_\ell_1}(\mathcal A,\omega_1,\ldots,\omega_d) 
		\tag{V1}
	\end{equation}
	
	1. Unfold $\mathcal A$ to the matrix $A_1 = \texttt{reshape}(\mathcal A, n_1,\prod^d_{j=2}n_j)\in\mathbb R^{n_1\times \prod^d_{j=2}n_j}$ and find the solution to
	\[
	( {\mathbf x}^*_1, {\mathbf y}^*_1) \in \arg\max_{(\mathbf x_1,\mathbf y_1)\in\mathbb R^{n_1}\times \mathbb R^{ \prod^d_{j=2}n_j },\bignorm{\mathbf x_1}=\bignorm{\mathbf y_1}=1}\nolimits \mathbf x_1^\top A_1\mathbf y_1;
	\]
	denote $\mathbf x^{\omega_1}_1 := N(\mathbf x^*_1,\omega_1)$, where $N(\cdot,\cdot)$ was defined in Lemma \ref{lem:l1_prox_constrained}.
	
	2. For $j=2,\ldots,d-1$, denote  matrices $A_j := \texttt{reshape}(A_{j-1}^\top\mathbf x_{j-1}^{\omega_{j-1}},n_j,\prod^d_{k=j+1}n_k )\in\mathbb R^{n_j\times \prod^d_{k=j+1}n_k}$ and find the solution to
	\[
	( {\mathbf x}^*_j, {\mathbf y}^*_j) \in \arg\max_{(\mathbf x_j,\mathbf y_j)\in\mathbb R^{n_j}\times \mathbb R^{ \prod^d_{k=j+1}n_k },\bignorm{\mathbf x_j}=\bignorm{\mathbf y_j}=1} \nolimits\mathbf x_j^\top A_j\mathbf y_j;
	\]
	denote $\mathbf x^{\omega_j}_j := N(\mathbf x^*_j,\omega_j)$.
	
	3. Denote  $\mathbf x^*_d:= A_{d-1}^{\top}\mathbf x^{\omega_{d-1} }_{d-1}/\normalnorm{A_{d-1}^{\top}\mathbf x^{\omega_{d-1}}_{d-1}}\in\mathbb R^{n_d}$ and  compute ${\mathbf x}^{\omega_d}_d:=N(\mathbf x^*_d,\omega_d)$.
	
	4. Return $(\mathbf x^{\omega_1}_1,\ldots,\mathbf x^{\omega_d}_d)$.
\end{boxedminipage}

\begin{boxedminipage}{0.92\textwidth}\small
	\begin{equation}  \label{proc:init4_order_d}
		\noindent {\rm Algorithm}~ (\mathbf x^{\omega_1}_1,\ldots,\mathbf x^{\omega_d}_d) =  {\rm approx\_alg\_\ell_1}(\mathcal A,\omega_1,\ldots,\omega_d) 
		\tag{V2}
	\end{equation}
	
	1. Unfold $\mathcal A$ to $A_1 = \texttt{reshape}(\mathcal A, n_1,\prod^d_{j=2}n_j)$ and denote $A_1^{\bar k}$ as the row of $A_1$ having the largest magnitude, i.e., $\normalnorm{ A_1^{\bar k}}=\max_{1\leq k\leq n_1}\bigfnorm{A_1^k}$. Let $\mathbf y_1^*=\normalxiaokuohao{ A_1^{\bar k}}^{\top}$ and $
	{\mathbf x}^*_1  =  A_1\mathbf y^*_1/\bignorm{A_1\mathbf y^*_1}\in\mathbb R^{n_1}$;
	denote $\mathbf x^{\omega}_1 := N(\mathbf x^*_1,\omega_1)$.
	\\
	
	2. For $j=2,\ldots,d-1$, denote $A_j = \texttt{reshape}(A_{j-1}^\top\mathbf x_{j-1}^{\omega},n_j,\prod^d_{k=j+1}n_k )$ and denote $A_j^{\bar k}$ as the row of $A_j$ having the largest magnitude. Let  $\mathbf y^*_j=\normalxiaokuohao{ A_j^{\bar k}}^{\top}$ and $
	{\mathbf x}^*_j  =  A_j\mathbf y^*_j/\bignorm{A_j\mathbf y^*_j}\in\mathbb R^{n_j}$;
	denote $\mathbf x^{\omega_j}_j := N(\mathbf x^*_j,\omega_j)$.\\
	
	3. Denote  $\mathbf x^*_d:= A_{d-1}^{\top}\mathbf x^{\omega_{d-1} }_{d-1}/\normalnorm{A_{d-1}^{\top}\mathbf x^{\omega_{d-1}}_{d-1}} \in\mathbb R^{n_d}$ and  compute ${\mathbf x}^{\omega_d}_d:=N(\mathbf x^*_d,\omega_d)$.
	
	4. Return $(\mathbf x^{\omega_1}_1,\ldots,\mathbf x^{\omega_d}_d)$.
\end{boxedminipage}

It can be checked that the computational complexity of Algorithm \ref{proc:init3_order_d} is dominated by $(n_1^2 n_2\cdots n_d)$ and that of Algorithm \ref{proc:init4_order_d} is dominated by $O(n_1\cdots n_d)$, showing the well-scalability of the proposed algorithms. 

In fact, the above two algorithms inherit the ideas of    \cite[Alg. C and D]{mao2022several} for the $\ell_0$-constrained problem. The differences  lie in the sparsification: For the non-sparse vectors $\mathbf x_j$'s, \cite{mao2022several} sparsifies them by finding a normalized sparse vector that is closest to $\mathbf x_j$, i.e., solving $\max_{\|\mathbf x\|=1,\|\mathbf x\|_0\leq r}\innerprod{\mathbf x_j}{\mathbf x}$, while here we solve \eqref{prob:sparsify_problem}. Comparing the $\ell_1$-regularized problem    \eqref{prob:str1approx_org_max_l1_regularized} with the $\ell_0$-constrained one \eqref{prob:str1approx_org_max_l0}, such a modification is natural. However, due to the $\ell_1$-regularization term, deriving the approximation bounds is more involved than \cite{mao2022several}, which needs more refined analysis. These will be detailed in the coming subsection.  

\subsection{Approximation bounds analysis}
We first lower bound the sparsification \eqref{prob:sparsify_problem}. Observe that in the algorithms, the vectors to be sparsified are always normalized; this property motivates us to study the following problem with $\omega\geq 0$: 
\begin{equation} \label{prob:min_sum_max_st_sphere}
	\min_{\mathbf x\in\mathbb R^n} ~ f(\mathbf x):=\sum^n_{i=1}\bigxiaokuohao{\bigjueduizhi{ \mathbf{x}(i)}-\omega}_+^2,~{\rm s.t.}~ \mathbf x^{\top}\mathbf x = 1.
\end{equation}
%where $\omega\geq 0$. and $(\cdot)_+ := \max\bigdakuohao{\cdot,0}$. 
Denote $\xi(n)$ as the minimum of \eqref{prob:min_sum_max_st_sphere}. If $\omega \geq  1/\sqrt{n}$, then taking $ \mathbf{x}(i) = 1/\sqrt{n}$ for each $i$,
it is seen that $\xi(n)=0$. When $\omega<1/\sqrt n$, we have the following estimation.
%Denote $k$ the number of nonzero entries of the minimizer $\mathbf x$ of \eqref{prob:min_sum_max_st_sphere}. We have the following lower bound on $\xi(n)$.
\begin{lemma}
	\label{lem:min_sum_max_st_sphere}
	Let $0< \omega <1/\sqrt{n}$. Then the minimum $\xi(n)$ of \eqref{prob:min_sum_max_st_sphere} satisfies
	\[
	\xi(n) %= k\bigxiaokuohao{1/\sqrt{k} - \omega}^2 
	\geq n\bigxiaokuohao{1/\sqrt{n} - \omega}^2. 
	\]
\end{lemma} 
\begin{proof}
	Note that the objective   $f(\mathbf x)$ is differentiable, whose partial derivative with respect to each entry $\mathbf{x}(i)$ is given by
	\[
	\frac{\partial f }{\partial \mathbf{x}(i)} = 2{\rm sgn}(\mathbf{x}(i))\bigxiaokuohao{ \bigjueduizhi{\mathbf{x}(i)} - \omega  }_+.
	\] 
	The Lagrangian function of \eqref{prob:min_sum_max_st_sphere} is given by $L(\mathbf x) = f(\mathbf x) - \lambda(\mathbf x^{\top}\mathbf x-1)$ with $\lambda\in\mathbb R$, and
	the KKT system of \eqref{prob:min_sum_max_st_sphere} reads as
	\begin{equation}
		\label{eq:kkt_min_sum_max_st_sphere}
		{\rm sgn}(\mathbf{x}(i))\bigxiaokuohao{\bigjueduizhi{\mathbf{x}(i)}-\omega}_+ = \lambda \mathbf{x}(i),~i=1,\ldots,n,~\mathbf x^{\top}\mathbf x= 1.
	\end{equation}
	Multiplying ${\rm sgn}(\mathbf{x}(i))$ on both sides of the first relation of \eqref{eq:kkt_min_sum_max_st_sphere} gives
	\begin{equation}
		\label{eq:kkt_min_sum_max_st_sphere_main}
		\bigxiaokuohao{\bigjueduizhi{\mathbf{x}(i)}-\omega}_+ = \lambda|\mathbf{x}(i)|,~i=1,\ldots,n,~\mathbf x^{\top}\mathbf x= 1.
	\end{equation}
	Since every minimizer of \eqref{prob:min_sum_max_st_sphere} satisfies \eqref{eq:kkt_min_sum_max_st_sphere_main}, we analyze the solution property of \eqref{eq:kkt_min_sum_max_st_sphere_main}. Let $\mathbf x$ solve \eqref{eq:kkt_min_sum_max_st_sphere_main}. If $\bigjueduizhi{\mathbf{x}(i)}>\omega$, then $(\bigjueduizhi{\mathbf{x}(i)}-\omega)_+ = \bigjueduizhi{\mathbf{x}(i)} - \omega>0$, which together with \eqref{eq:kkt_min_sum_max_st_sphere_main} gives  
	\begin{equation}
		\label{eq:proof_min_sum_max_st_sphere_1}
		\lambda>0,~\bigjueduizhi{\mathbf{x}(i)} = \frac{\omega}{1-\lambda},~\forall i \in \{ i\mid \bigjueduizhi{\mathbf{x}(i)}>\omega  \}. 
	\end{equation}
	If
	$\bigjueduizhi{\mathbf{x}(i)} \leq \omega$, then $(\bigjueduizhi{\mathbf{x}(i)}-\omega)_+=0$, which in connection with $\lambda>0$ and \eqref{eq:kkt_min_sum_max_st_sphere_main} yields that
	\begin{equation}
		\label{eq:proof_min_sum_max_st_sphere_2}
		\mathbf{x}(i)=0,~\forall i\in \{ i\mid \bigjueduizhi{\mathbf{x}(i)}\leq \omega  \}. 
	\end{equation}
	Note that if $\mathbf x$ is a minimizer to \eqref{prob:min_sum_max_st_sphere}, then it must obey \eqref{eq:proof_min_sum_max_st_sphere_1} and \eqref{eq:proof_min_sum_max_st_sphere_2}. Let $k$ be the number of nonzero entries of a minimizer $\mathbf x$. Since $\bignorm{\mathbf x}=1$,  the above analysis   implies that      $\bigjueduizhi{\mathbf{x}(i)}= 1/\sqrt{k}\geq \omega, \forall i\in\{ i\mid \bigjueduizhi{\mathbf{x}(i)}>\omega  \}$, and so
	\[
	\xi(n) = \sum^n_{i=1}\nolimits\bigxiaokuohao{\bigjueduizhi{\mathbf{x}(i)}-\omega}_+^2= k\bigxiaokuohao{1/\sqrt{k} -\omega}^2.
	\] 
	We then consider the function $g(y):= y\bigxiaokuohao{1/\sqrt{y} -\omega}^2$, where $1\leq y\leq n$. Since  $\omega <1/\sqrt{n}$, $g^{\prime}(y) = \omega(\omega-1/\sqrt{y})<0$, i.e., $g(\cdot)$ is nonincreasing, leading to that
	\[
	k\bigxiaokuohao{1/\sqrt{k} -\omega}^2 \geq n\bigxiaokuohao{1/\sqrt{n}-\omega}^2,~ k\leq n.
	\]
	The proof has been completed. 
\end{proof}

Let $\lambda_{\max}(\cdot)$ denote the largest singular value of a given matrix. The following   lemmas are important.
\begin{lemma}
	\label{lem:lower_bound_lem:1}
	Let  $A\in\mathbb R^{n\times m}$ be nonzero, where $(\mathbf x,\mathbf y)$ is a normalized singular vector pair corresponding to $\lambda_{\max}(A)$. Denote  $\mathbf x^w := N(\mathbf x,\omega)$ with $N(\cdot,\cdot)$ defined in \eqref{sol:subproblem}. If $\omega <1/\sqrt{n}$, then 
	\[\small
	\bignorm{A^{\top}\mathbf x^{\omega}} \geq (\sqrt{\xi(n)}+\omega)\lambda_{\max}(A)\geq (1-\omega\sqrt{n} + \omega)\lambda_{\max}(A).
	\]
\end{lemma}
\begin{proof}
	Let $\lambda_1=\lambda_{\max}(A)>0$ provided that $A\neq 0$. Since $\bignorm{A\mathbf y}=\lambda_1$, $A^{\top}\mathbf x = \lambda_1\mathbf y$ and $A\mathbf y = \lambda_1\mathbf x$, we have
	\begin{eqnarray}
		\bignorm{A^{\top}\mathbf x^{\omega}} &=& \lambda_1^{-1}\bignorm{A^{\top}\mathbf x^{\omega}}\cdot \bignorm{A^{\top} {\mathbf x}}\nonumber\\
		&\geq & \lambda_1^{-1} \innerprod{A^{\top}\mathbf x^{\omega}}{A^{\top} {\mathbf x}} =     \innerprod{A^{\top}\mathbf x^{\omega}}{ {\mathbf y}}\nonumber\\
		&=&  \lambda_1    \innerprod{\mathbf x^{\omega}}{ {\mathbf x}} \nonumber\\
		&=& \lambda_1\bigxiaokuohao{   \innerprod{\mathbf x^{\omega}}{ {\mathbf x}}  - \omega\bigonenorm{\mathbf x^{\omega}}    } + \lambda_1\omega \bigonenorm{\mathbf x^{\omega}}. \label{eq:proof:lem:lower_bound_lem:1}
	\end{eqnarray}
	The assumption     $\omega<1/\sqrt{n}$  with the fact that $\bignorm{\mathbf x}=1$  shows that $\bignorm{S(\mathbf x,\omega)}\neq 0$, which together with the definition of $N(\cdot,\cdot)$ gives that $\mathbf x^{\omega} = S(\mathbf x,\omega)/\bignorm{S(\mathbf x,\omega)}$; and Lemma \ref{lem:l1_prox_constrained} tells us that 
	\begin{eqnarray}
		\innerprod{\mathbf x}{\mathbf x^{\omega}} - \omega\bigonenorm{\mathbf x^{\omega}} &=& \sqrt{ \sum^n_{i=1}\nolimits \bigxiaokuohao{ \bigjueduizhi{\mathbf{x}(i)}-\omega }_+^2 }\nonumber\\
		&\geq& \sqrt{\min_{\mathbf z\in\mathbb R^n,\bignorm{\mathbf z}=1}  { \sum^n_{i=1}\nolimits \bigxiaokuohao{ \bigjueduizhi{\mathbf{z}(i)} - \omega  }_+^2 }}\nonumber\\
		& =&\sqrt{\xi(n)}\geq (1 -\omega\sqrt{n}), \label{eq:proof:lem:lower_bound_lem:2}
	\end{eqnarray}
	where the last inequality follows from Lemma \ref{lem:min_sum_max_st_sphere}. On the other hand, 
	\begin{equation}
		\bigonenorm{\mathbf x^{\omega}} \geq \min_{\mathbf z\in\mathbb R^n,\bignorm{\mathbf z}=1}\nolimits\bigonenorm{\mathbf z} = 1. \label{eq:proof:lem:lower_bound_lem:3}
	\end{equation}
	\eqref{eq:proof:lem:lower_bound_lem:1}, \eqref{eq:proof:lem:lower_bound_lem:2} and \eqref{eq:proof:lem:lower_bound_lem:3} then yield $\bignorm{A^{\top}\mathbf x^{\omega}} \geq (1-\omega\sqrt{n} + \omega)\lambda_{\max}(A)$, as desired.  
\end{proof}

\begin{lemma}
	\label{lem:proc4}
	Let $A\in\mathbb R^{n\times m}$ be nonzero; let $A^{\bar k}$ be the row of $A$ having the largest magnitude. Denote $\mathbf y=(A^{\bar k})^\top\in\mathbb R^m$, $\mathbf x = A\mathbf y/\normalnorm{A\mathbf y}$, and $\mathbf x^w :=N(\mathbf x,\omega)$. If $\omega <1/\sqrt{n}$, then 
	\[
	\bignorm{A^\top\mathbf x^{\omega}} \geq \frac{1-\omega\sqrt n+\omega}{\sqrt n}\bigfnorm{A}.
	\]
\end{lemma}
\begin{proof}
	Using the relation $\normalnorm{\mathbf w} = \max_{\normalnorm{\mathbf z}=1}\innerprod{\mathbf w}{\mathbf z}$ for any vector $\mathbf w$,	we have
	\begin{eqnarray*}
		\bignorm{A^\top \mathbf x^w} &=& \max_{\normalnorm{\mathbf z}=1 }\nolimits \innerprod{A^\top\mathbf x^w }{\mathbf z} \nonumber\\
		&\geq& \innerprod{A^\top \mathbf x^w}{\mathbf y/\normalnorm{\mathbf y}} = \innerprod{\mathbf x^{\omega}}{ A\mathbf y/\normalnorm{\mathbf y}} = \frac{\normalnorm{A\mathbf y}}{\normalnorm{\mathbf y}} \innerprod{\mathbf x^{\omega}}{\mathbf x}\nonumber  \\
		&=& \frac{\normalnorm{A\mathbf y}}{\normalnorm{\mathbf y}} \bigxiaokuohao{\innerprod{\mathbf x^{\omega}}{\mathbf x/\normalnorm{\mathbf x}} - \omega\bigonenorm{\mathbf x^{\omega}}} + \omega\frac{\normalnorm{A\mathbf y}}{\normalnorm{\mathbf y}} \bigonenorm{\mathbf x^{\omega}} \nonumber\\
		&\geq& \frac{\normalnorm{A\mathbf y}}{\normalnorm{\mathbf y}} \bigxiaokuohao{1-\omega\sqrt n + \omega}, \label{eq:proof:proc4:1}
	\end{eqnarray*}
	where the last inequality follows from $\mathbf x^{\omega} = S(\mathbf x,\omega)$ due to that $\omega <1/\sqrt n$ and $\normalnorm{\mathbf x}=1$, and from \eqref{eq:proof:lem:lower_bound_lem:2} and \eqref{eq:proof:lem:lower_bound_lem:3}. On the other hand, 
	\begin{eqnarray*}
		\normalnorm{A\mathbf y}^2/\normalnorm{\mathbf y}^2 &=& \sum^n_{k=1}\nolimits  \bigxiaokuohao{A^k\mathbf y}^2/\normalnorm{\mathbf y}^2   \geq \bigxiaokuohao{A^{\bar k}\mathbf y}^2/\normalnorm{\mathbf y}^2    = \normalnorm{A^{\tilde k}}^2 \geq \frac{1}{n}\bigfnorm{A}^2, \label{eq:proof:proc4:2}
	\end{eqnarray*}
	where the last inequality comes from the definition of $A^{\bar k}$. The required result follows. 
\end{proof}

\begin{lemma}
	\label{prop:welldefined30} 
	If $\mathcal A\neq 0$  and $(\mathbf x^{\omega_1}_1,\ldots,\mathbf x^{\omega_d}_d)$ is generated by Algorithms \ref{proc:init3_order_d} or \ref{proc:init4_order_d}, then $A_j\neq 0$ and $\mathbf x^{\omega_j}_j\neq 0$, $1\leq j\leq d$.
\end{lemma}
\begin{proof} We only prove the results for Algorithm \ref{proc:init3_order_d} while that for Algorithm \ref{proc:init4_order_d} is analogous. It follows from $\mathcal A\neq 0$ that $ {\mathbf x}_1\neq 0$ and $ {\mathbf y}_1\neq 0$. Lemma \ref{lem:l1_prox_constrained} shows that $\mathbf x^{\omega_1}_1\neq 0$. Then, $\innerprod{A_1^{\top}\mathbf x_1^{\omega}}{\mathbf y_1} = \lambda_{\max}(A_1)\innerprod{\mathbf x_1}{\mathbf x^{\omega}_1}>0$, where the strict inequality also comes from Lemma \ref{lem:l1_prox_constrained}. Thus $A^{\top}_1\mathbf x^{\omega_1}_1\neq 0$, and so $A_2\neq 0$. Similar argument can be applied to show that $A_j\neq 0$ and $\mathbf x^{\omega_j}_j\neq 0$ for each $j$.  
\end{proof}

%The following proposition is also used in deriving the lower bound.
%\begin{proposition}
%	\label{prop:relation_A1_spec_norm_vopt}
%	Let $A_1$ be defined as in Algorithm \ref{proc:init3}. Then $\lambda_{\max}(A_1)\geq v\eqref{prob:str1approx_org_max}$.
%\end{proposition}
%\begin{proof}
%Let $(\mathbf x_1^*,\ldots,\mathbf x_d^*)$ be optimal to \eqref{prob:str1approx_org_max}.	It is clear that $\lambda_{\max}(A_1) \geq \innerprod{\mathcal A}{\mathbf x_1^*\circ\cdots\circ\mathbf x_d} \geq v\eqref{prob:str1approx_org_max}$. 
%\end{proof}

Now we are in the position to analyze the approximation bound. To simplify notations, we denote
\[
\mathcal A\mathbf x_1\cdots\mathbf x_d:=\innerprod{\mathcal A}{\mathbf x_1\circ\cdots\circ \mathbf x_d}
\]
in the sequel. In fact, it is more reasonable to  derive the approximation bound on $\mathcal A\mathbf x_1\cdots\mathbf x_d$ than on $\mathcal A\mathbf x_1\cdots\mathbf x_d-\sum^d_{j=1}\omega_j\bigonenorm{\mathbf x_j}$, 
because the goal of the studied problem is  to approximate $\bigfnorm{\mathcal A-\lambda\mathbf x_1\circ\cdots\circ\mathbf x_d}$ essentially, 
while it is clear that the larger the   $\mathcal A\mathbf x_1\cdots\mathbf x_d$ is, the better the approximation will be. 
In the following, let $(\tilde{\mathbf x}_1,\ldots,\tilde{\mathbf x}_d)$ be any feasible point to \eqref{prob:str1approx_org_max_l1_regularized} or \eqref{prob:str1approx_org_max_l0}, i.e., 
$(\tilde{\mathbf x}_1,\ldots,\tilde{\mathbf x}_d)\in \bigdakuohao{ (\mathbf x_1,\ldots,\mathbf x_d)\mid \|{\mathbf x_j}\|=1,1\leq j\leq d }$.  
\begin{theorem}
	\label{th:init_theory_bound_3}
	Let $d\geq 3$ and let $(\mathbf x^{\omega_1}_1,\ldots,\mathbf x^{\omega_d}_d)$ be generated by Algorithm \ref{proc:init3_order_d}. If  $\omega_j<1/\sqrt{n_j}$, $1\leq j\leq d$,   then 
	\begin{align*}
		\mathcal A\mathbf x^{\omega_1}_1\cdots\mathbf x^{\omega_d}_d &\geq  \frac{ \prod^d_{j=1}\bigxiaokuohao{1-\omega_j\sqrt{n_j} +\omega_j }   }{ \sqrt{ \prod^{d-1}_{j=2} n_j } } \lambda_{\max}(A_1) \\
		&\geq \frac{ \prod^d_{j=1}\bigxiaokuohao{1-\omega_j\sqrt{n_j} +\omega_j }   }{ \sqrt{ \prod^{d-1}_{j=2} n_j } } \mathcal A\tilde{\mathbf x}_1\cdots\tilde{\mathbf x}_d.
	\end{align*}
\end{theorem}
\begin{remark}
	To prove the theorem, by noticing the definitions of $A_j$'s and using the Kronecker representation,  we will use the following fact
	\begin{align}\label{eq:relation}
		\mathcal A\mathbf x^{\omega_1}_1\cdots\mathbf x^{\omega_d}_d &= \innerprod{A_1^{\top}\mathbf x^{\omega_1}_1}{\mathbf x^{\omega_d}_d\otimes\cdots\otimes \mathbf x^{\omega_2}_2} = \innerprod{A_2^{\top}\mathbf x^{\omega_2}_2}{\mathbf x^{\omega_d}_d\otimes\cdots\otimes\mathbf x^{\omega_3}_3} \nonumber \\
		&=\cdots = \innerprod{A_{d-1}^{\top}\mathbf x^{\omega_{d-1}}_{d-1}  }{\mathbf x^{\omega_d}_d},
	\end{align}
	where $\otimes$ denotes the Kronecker product. 
\end{remark}
\begin{proof}
	Using the above representation and recalling the definitions of $\mathbf x_j^{\omega_j}$ and $\mathbf x_j^*$ in the algorithm,  we first rewrite $\mathcal A\mathbf x^{\omega_1}_1\cdots\mathbf \mathbf x^{\omega_d}_d $ as $\innerprod{A^{\top}_{d-1}\mathbf x^{\omega_{d-1}}_{d-1}}{\mathbf x^{\omega_d}_d}$. It follows from $\omega_d< 1/\sqrt{n_d}$ and   $\normalnorm{\mathbf x_d^*}=1$  that $S(\mathbf x_d^*,\omega_d)\neq 0$, and so Lemma \ref{lem:l1_prox_constrained} gives $\innerprod{\mathbf x^*_d}{\mathbf x^{\omega_d}_d}-\omega_d\bigonenorm{\mathbf x^{\omega}_d} = \sqrt{ \sum^{n_d}_{i=1} \bigxiaokuohao{\bigjueduizhi{ \mathbf x^*_{d}(i) } - \omega_d}_+^2 }$. Then we have %This together with $\normalnorm{\mathbf x_j^*}=1$ and Lemma \ref{lem:min_sum_max_st_sphere} yields
	\begin{equation}
		\label{eq:proof:alg_A_lower_bound_d3:1}
		\begin{split}
			\mathcal A\mathbf x^{\omega_1}_1\cdots\mathbf \mathbf x^{\omega_d}_d  & \overset{{\rm by}~ \eqref{eq:relation}}{=}\innerprod{A^{\top}_{d-1}\mathbf x^{\omega_{d-1}}_{d-1}}{\mathbf x^{\omega_d}_d} = \normalnorm{A_{d-1}^{\top}\mathbf x^{\omega_{d-1}}_{d-1}}\innerprod{\mathbf x^*_d}{\mathbf x^{\omega_d}_d}\\
			&~~~~ \geq\bigxiaokuohao{ 1-\omega_d\sqrt{n_d} + \omega_d  } \normalnorm{A_{d-1}^{\top}\mathbf x^{\omega_{d-1}}_{d-1}},
		\end{split}
	\end{equation}
	where  the second equality follows from the definition of $\mathbf x_d^*$, and    Lemma \ref{lem:min_sum_max_st_sphere} together with $\normalnorm{\mathbf x_j^*}=1$ yields the inequality. 
	
	We then build the relation between $\bignorm{A_j^{\top}\mathbf x^{\omega_j}_j}$ and $\bignorm{A_{j-1}^{\top}\mathbf x^{\omega_{j-1}}_{j-1}}$ for $j=d-1,\ldots, 2$.  Lemma \ref{prop:welldefined30} shows that $A_j\neq 0$ . It then follows   from the range of $\omega_j$, the definition of $\mathbf x^{\omega_j}_j$, and Lemma \ref{lem:lower_bound_lem:1} that  $\bignorm{A_{j}^{\top}\mathbf x^{\omega_j}_j} \geq \lambda_{\max}(A_j)(1-\omega_j\sqrt{n_j}+\omega_j)$, and so  
	\begin{equation}
		\begin{split}
			\bignorm{A_j^{\top}\mathbf x^{\omega_j}_j}& \geq \lambda_{\max}(A_j)(1-\omega_j\sqrt{n_j}+\omega_j) \geq \frac{1-\omega_j\sqrt{n_j}+\omega_j}{\sqrt n_j}\bigfnorm{A_j} \\
			&= \frac{1-\omega_j\sqrt{n_j}+\omega_j}{\sqrt n_j}\bignorm{A_{j-1}^{\top}\mathbf x^{\omega_{j-1}}_{j-1}},
		\end{split}
	\end{equation}
	where the second inequality uses that $\lambda_{\max}(A_j)/\bigfnorm{A_j}\geq 1/\sqrt{n_j}$ and the equality follows from that the reshape operation does not change the norm size. Finally, it  follows from the range of $\omega_1$ and Lemma \ref{lem:lower_bound_lem:1} that $\normalnorm{A_1^{\top}\mathbf x^{\omega_1}_1} \geq (1-\omega_1\sqrt{n_1}+\omega_1)\lambda_{\max}(A_1)$. The above analysis together with the simple fact that $\lambda_{\max}(A_1)\geq \mathcal A\tilde{\mathbf x}_1\cdots\tilde{\mathbf x}_d$ gives the desired bound. 
\end{proof}

In particular, if $\omega_j = O(1/\sqrt{n_j})$ where the constant behind the big $O$ is strictly less than $1$, then the order of the ratio is $O\bigxiaokuohao{ \bigxiaokuohao{  \prod^{d-1}_{j=2}n_j    }^{-1/2} }  $, which is of the same order as \cite[Alg. 1]{he2010approximation} in the non-sparse rank-1 approximation setting.

The approximation bound for Algorithm \ref{proc:init4_order_d} is given as follows. 

\begin{theorem}
	\label{th:init_theory_bound_4_general}
	Let $(\mathbf x^{\omega_1}_1,\ldots,\mathbf x^{\omega_d}_d)$ be generated by Algorithm \ref{proc:init4_order_d}. If $\omega_j<1/\sqrt{n_j}$, $1\leq j\leq d$, then
	\begin{small}
		\begin{align*}
			\mathcal A\mathbf x^{\omega_1}_1\cdots\mathbf x^{\omega_d}_d&\geq \frac{\prod^d_{j=1}\bigxiaokuohao{1-\omega_j\sqrt{n_j} + \omega_j}   }{\sqrt{ \prod^{d-1}_{j=1} n_j }}\bigfnorm{\mathcal A} \geq \frac{\prod^d_{j=1}\bigxiaokuohao{1-\omega_j\sqrt{n_j} + \omega_j}   }{\sqrt{ \prod^{d-1}_{j=1} n_j }}\lambda_{\max}(A_1)\\
			& \geq \frac{\prod^d_{j=1}\bigxiaokuohao{1-\omega_j\sqrt{n_j} + \omega_j}   }{\sqrt{ \prod^{d-1}_{j=1} n_j }} \mathcal A\tilde{\mathbf x}_1\cdots\tilde{\mathbf x}_d.
		\end{align*}
	\end{small}
\end{theorem}
\begin{proof}
	Similar to the proof of Theorem \ref{th:init_theory_bound_3}, the definition of $\mathbf x^{\omega_d}_d$ shows that
	\begin{equation*}
		\label{eq:proof:alg_B_lower_bound_d3:1}
		\mathcal A\mathbf x^{\omega_1}_1\cdots\mathbf \mathbf x^{\omega_d}_d   =\innerprod{A^{\top}_{d-1}\mathbf x^{\omega_{d-1}}_{d-1}}{\mathbf x^{\omega_d}_d}  \geq\bigxiaokuohao{ 1-\omega_d\sqrt{n_d} + \omega_d  } \normalnorm{A_{d-1}^{\top}\mathbf x^{\omega_{d-1}}_{d-1}}. 
	\end{equation*}
	For $j=d-1,\ldots,2$,  since $A_j$'s are nonzero by Lemma \ref{prop:welldefined30},   according to Lemma \ref{lem:proc4}, the definition of $\mathbf x^{\omega_j}_j$, and the range of $\omega_j$, we have
	\begin{align*}
		\label{eq:proof:alg_B_lower_bound_d3:2}
		\normalnorm{A_{j}^{\top}\mathbf x^{\omega_{j}}_{j}} \geq \frac{1-\omega_j\sqrt{n_j}+\omega_j}{\sqrt{n_j}}\bigfnorm{A_j} =  \frac{1-\omega_j\sqrt{n_j}+\omega_j}{\sqrt{n_j}} \bignorm{A_{j-1}^{\top}\mathbf x^{\omega_{j-1}}_{j-1}}.
	\end{align*}
	Finally, again by Lemma \ref{lem:proc4}, the definition of $\mathbf x^{\omega_1}_1$, and the range of $\omega_1$, we obtain $\bignorm{A_1^{\top}\mathbf x^{\omega_1}_1} \geq \frac{1-\omega_1\sqrt{n_1}+\omega_1}{\sqrt{n_1}}\bigfnorm{\mathcal A}$. Combining the above analysis gives the result.  
\end{proof}

\begin{remark}
	
	The analysis in this section suggests that for the regularization parameters, we can choose $\omega_j<1/\sqrt{n_j}$. Empirically, we find that $\omega_j$ closing to $1/\sqrt{n_j}$ gives better results, and in the experiments we always set $\omega_j = 1/\sqrt{n_j}-10^{-5}$.
\end{remark}

%Before ending this section, we summarize the approximation ratio and computational complexity of the proposed algorithms in Table \ref{tab:approx_alg_summary}. For convenience we set $r_1=\cdots=r_d$ and $n_1=\cdots = n_d$.
%
%\begin{table}[htbp]
%	\centering
%	\caption{\footnotesize   Comparisons of the proposed approximation algorithms   on the approximation bound and computational complexity.}
%	\begin{mytabular1}{ccc}
	%		\toprule
	%		\multicolumn{1}{c}{Algorithm} & Approximation bound &  Computational complexity  \\
	%		\toprule
	%	Algorithm	\ref{proc:init1_order_d} 	& $\frac{v^{\rm opt}}{\sqrt{ r^{d-1} }}$   &	$O(n^d\log(n) + dn\log(n) )$ \\
	%	\midrule
	%	Algorithm	\ref{proc:init2_order_d} 	& $\sqrt{\frac{r^2}{n^2}}\frac{v^{\rm opt}}{\sqrt{r^{d-2}}}$ &	$O(n^{d+1} + dn\log(n))$ \\
	%	\midrule
	%	Algorithm	\ref{proc:init3_order_d} 	& $\sqrt{\frac{r^d}{n^d}} \frac{\lambda_{\max}(A_1)}{\sqrt{n^{d-2}}}$ &	$O(n^{d+1} + dn\log(n))$ \\
	%	\midrule
	%		Algorithm	\ref{proc:init4_order_d} 	& $\sqrt{\frac{r^d}{n^d}} \frac{ \bigfnorm{\mathcal A}  }{\sqrt{n^{d-1}}}$ &	$O(n^{d} + dn\log(n))$ \\
	%		\bottomrule
	%	\end{mytabular1}%
%	\label{tab:approx_alg_summary}%
%\end{table}%

%Point 2 of Theorem \ref{th:convergence} means that if the initial point is generated by the approximation algorithms in Sect. \ref{sec:approx_alg}, then the sparsity will not be changed after running Algorithm \ref{alg:reweighted}. 
\section{Preliminary Numerical Examples}\label{sec:numer_experiments}
All the   computations are conducted on an Intel i7  CPU desktop computer with 32 GB of RAM. The supporting software is Matlab R2019b.   
The tensors are given by 
%\begin{equation}\label{exp:tr1a}
$\mathcal A = \sum^{10}_{i=1} \mathbf x_{1,i}\circ\cdots\circ\mathbf x_{d,i} \in\mathbb R^{n_1\times\cdots\times n_d}$, 
%\end{equation}
%where $\mathbf x_{j,i}\in\mathbb R^{n_j}$, $i=1,\ldots,R$, $j=1,\ldots,d$, and we let $R=10$.
where the vectors are first randomly drawn from the normal distribution, and then some of the entries are randomly set to be zero. The sparsity ratio of the tensor is denoted as $sr$.   We set $d=4$. For each case, we randomly generated $50$ instances, and the   averaged results are presented. To be more stable, we actually use  $\mathcal A/\|\mathcal A\|_{\infty}$ as the data tensor, and multiply the results by $\|\mathcal A\|_{\infty}$ after the computation; here $\|\cdot\|_{\infty}$ denotes the largest entry in magnitude. 
We always set $\omega_j = 1/\sqrt{n_j} - 10^{-5}$ for the algorithms.  The value
%\begin{equation}\label{eq:upper_bound}
$v^{\rm ub} := \min\{\lambda_{\max}(A_{(1)}),\ldots,\lambda_{\max}(A_{(d)})  \}$
%\end{equation}
is used as an upper bound, where $A_{(j)} $ is  the $j$-mode unfolding of $\mathcal A$. 

\begin{figure}[h] 
	\centering
	\includegraphics[height=3.5cm,width=13cm]{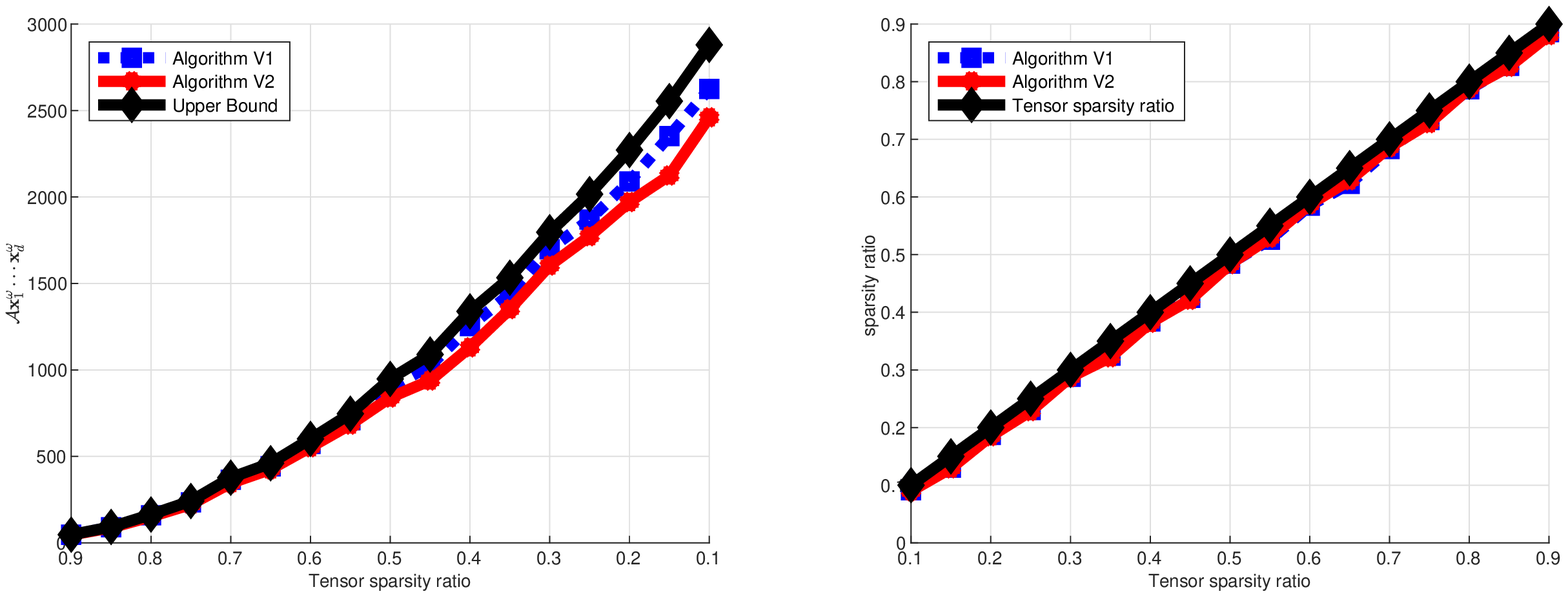}
	\caption{Comparisons of Algorithms \ref{proc:init3_order_d}  and \ref{proc:init4_order_d} %for solving \eqref{prob:str1approx_org_max_l1_regularized_inequality_constraint} 
		when $d=4$, $n_j=50$, and the sparsity ratio of the tensor $sr$ varies from $10\%$ to $90\%$. Left panel:  $\mathcal A\mathbf x^{\omega_1}_1\cdots\mathbf x^{\omega_d}_d$ versus the tensor sparsity ratio $sr$; right panel: the sparsity ratio of the output vectors versus the true tensor sparsity ratio $sr$.} 
	\label{fig:vary_sr} %% label for entire figure
\end{figure}

We first fix $n_j=50$ for each $j$ and let  the sparsity ratio of the tensor   $sr$ vary from    $10\%$ to $90\%$. The results are plotted in Fig. \ref{fig:vary_sr}. The results of Algorithm \ref{proc:init3_order_d} is in blue with square markers  while those of Algorithm \ref{proc:init4_order_d} is in red with star markers. The values $\mathcal A\mathbf x_1^{\omega_1}\cdots\mathbf x_d^{\omega_d}$ are plotted in the left panel, from which we observe that both algorithms perform well, as the results are close to the upper bound, which is in black with diamond markers. Algorithm \ref{proc:init3_order_d} is better because it is based on SVDs, which retain more information in the non-sparse candidate vectors $\mathbf x_j^*$.  The sparsity ratios of $\mathbf x_j^{\omega_j}$'s are plotted in the right panel, from which we observe that both algorithms are quite close to the true sparsity ratios, which is in black with diamond markers. This confirms the theoretical suggestion on choosing the regularization parameters in the previous section.

\begin{figure}[h] 
	\centering
	\includegraphics[height=3.5cm,width=13cm]{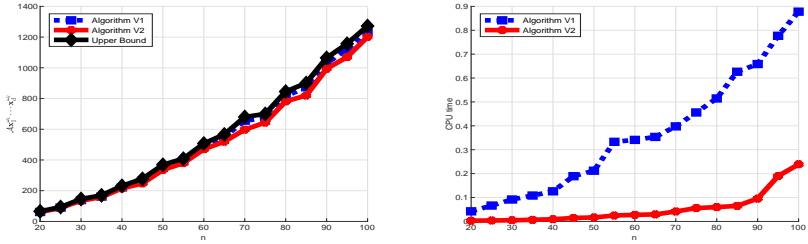}
	\caption{Comparisons of Algorithms \ref{proc:init3_order_d}  and \ref{proc:init4_order_d} %for solving \eqref{prob:str1approx_org_max_l1_regularized_inequality_constraint} 
		when $d=4$, $sr=70\%$, and $n_j$ varies from $20$ to $100$. Left panel:  $\mathcal A\mathbf x^{\omega_1}_1\cdots\mathbf x^{\omega_d}_d$ versus $n$; right panel: CPU time versus $n$.} 
	\label{fig:vary_n} %% label for entire figure
\end{figure}

We then fix $sr=70\%$ and let $n_j$'s vary from $20$ to $100$. The results are plotted in Fig. \ref{fig:vary_n}, where the left panel still shows $\mathcal A\mathbf x_1^{\omega_1}\cdots\mathbf x_d^{\omega_d}$ while the right one shows the CPU time. Considering $\mathcal A\mathbf x_1^{\omega_1}\cdots\mathbf x_d^{\omega_d}$, both algorithms still perform well when $n$ changes, and considering the CPU time, Algorithm \ref{proc:init4_order_d} is significantly better
because it does not need to compute SVDs.

\begin{figure}[h] 
	\centering
	\subfigure[$\mathcal A\mathbf x_1^{iter}\cdots \mathbf x_d^{iter}$ and CPU time.]{
		\label{fig:d3_approx} %% label for second subfigure
		\includegraphics[height=3.5cm,width=13cm]{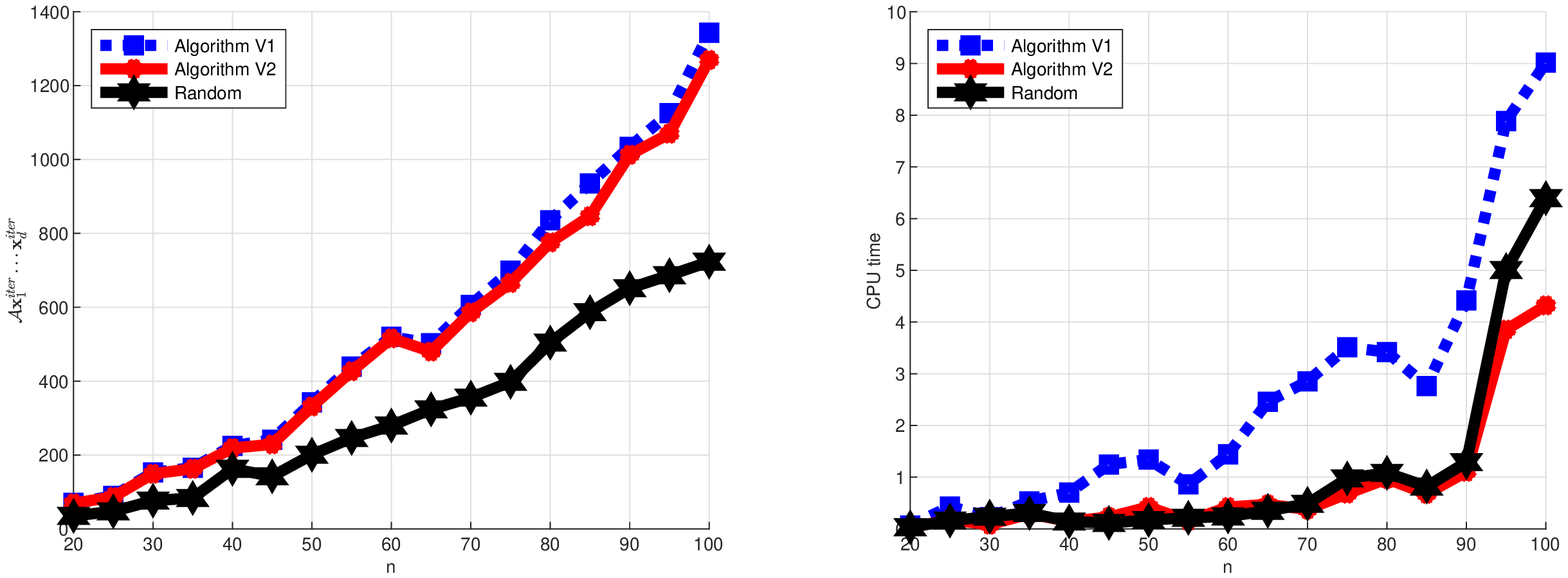}}\\
	\subfigure[Sparsity ratio of the output vectors.]{
		\label{fig:d4_approx}%% label for second subfigure
		\includegraphics[height=3.5cm,width=6cm]{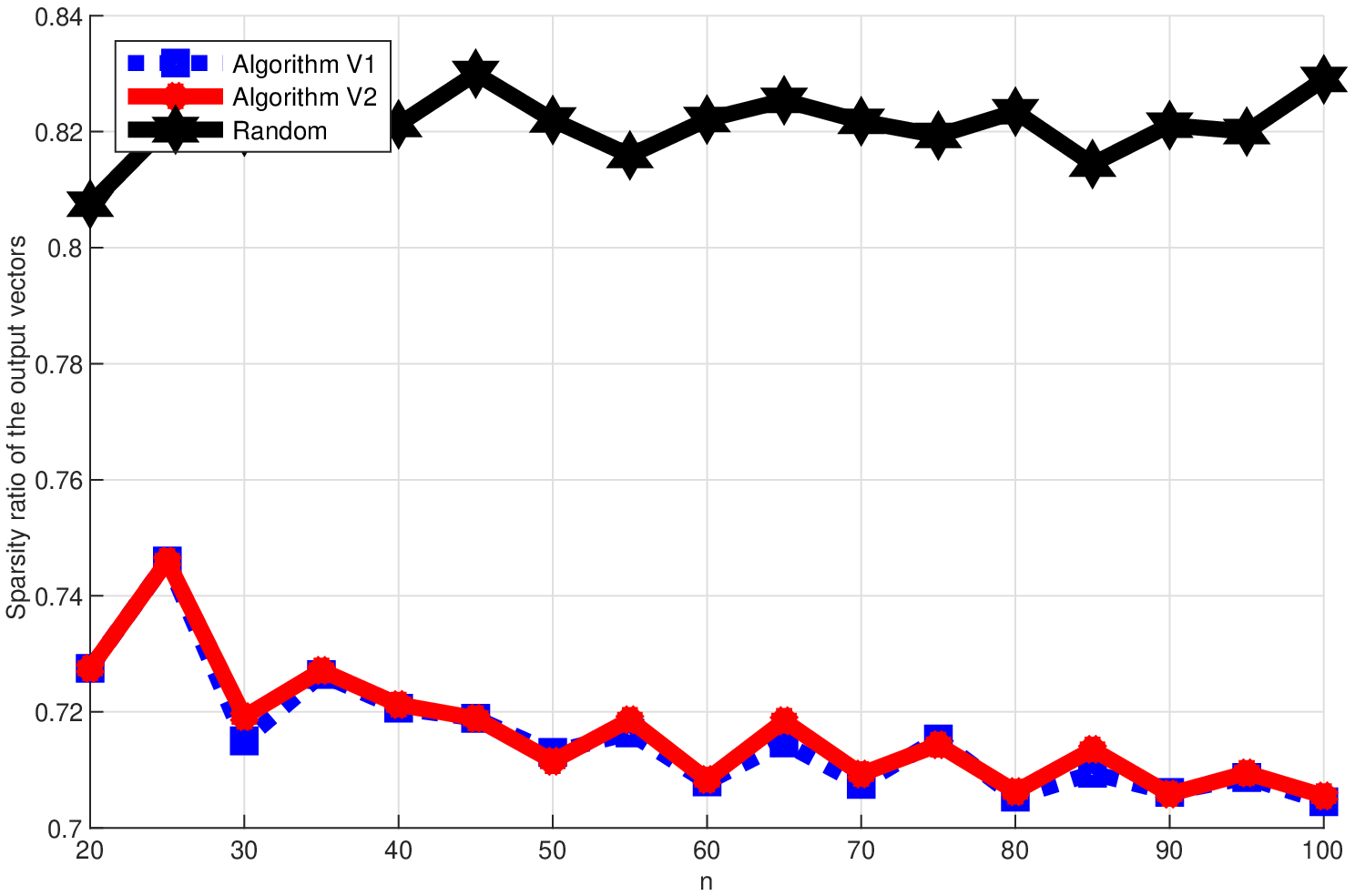}}\\   
	\caption{Comparisons of Algorithms \ref{proc:init3_order_d} and \ref{proc:init4_order_d} as initializations when $d=4$, $sr=70\%$, and $n_j$ varies from $20$ to $100$.} 
	\label{fig:iter} %% label for entire figure
\end{figure}

Finally, we   compare Algorithms \ref{proc:init3_order_d} and \ref{proc:init4_order_d} as initialization  procedures for the alternating maximization method (AMM) for solving \eqref{prob:str1approx_org_max_l1_regularized}. AMM is stopped if the distance between the successive two iterative points is smaller than $10^{-6}$.  AMM initialized by the random initialization is used as a baseline, which is generated as $N( \mathbf x/\|\mathbf x\|,\omega )$ where $\mathbf x$ is randomly drawn from the normal distribution. We plot the value $\mathcal A\mathbf x_1^{iter}\cdots\mathbf x^{iter}_d$ ($\mathbf x^{iter}_j$'s are outputted by AMM), the CPU time (counting both that of the approximation algorithm and AMM), and the sparsity ratio of the output vectors in Fig. \ref{fig:iter}. From the first subfigure, we see that the output value of AMM initialized  by Algorithm \ref{proc:init3_order_d} is the best one, followed by Algorithm \ref{proc:init4_order_d}, and both two are far more better than that initialized by the random initialization. From the second subfigure, we see that AMM initialized by  Algorithm \ref{proc:init4_order_d} is the most efficient one, followed by that initialized by the random initialization and that initialized by Algorithm \ref{proc:init3_order_d}. We do not plot the number of iterations as all the cases take about $4$ to $6$ iterations to converge.   From the third subfigure, we observe that the sparsity ratios of the output vectors of AMM initialized by Algorithm \ref{proc:init3_order_d} and \ref{proc:init4_order_d} are close to the true sparsity ratio $70\%$, which are far more better than AMM initialized by the random initialization. 

In summary, the above experiments preliminarily demonstrate the effectiveness and efficiency of the proposed algorithms, and confirm the theoretical suggestion on choosing the regularization parameters.

\section{Conclusions}\label{sec:conclusions}
By extending the idea of multilinear relaxation \cite{he2010approximation,mao2022several}, we devised two approximation algorithms for $\ell_1$-regularized rank-$1$ approximation to tensors. The algorithms are easily implemented and well scalable.  For a $d$-th order tensor, $\frac{ \prod^d_{j=1}\bigxiaokuohao{1-\omega_j\sqrt{n_j} +\omega_j }   }{ \sqrt{ \prod^{d-1}_{j=2} n_j } }$-- and $\frac{ \prod^d_{j=1}\bigxiaokuohao{1-\omega_j\sqrt{n_j} +\omega_j }   }{ \sqrt{ \prod^{d-1}_{j=1} n_j } }$--lower bounds were established when $\omega_j < 1/\sqrt{n_j}$, where $\omega_j$'s are the regularization parameters.     Numerical examples were provided to verify the algorithms and the effectiveness of setting $\omega_{j}=1/\sqrt{n_j}-10^{-5}$.  We also remark that similar ideas of approximation algorithms can be designed for the $\ell_0$-regularized and $\ell_1$-constrained cases, while the key is how to analyze their lower bounds.

{\scriptsize\section*{Acknowledgement}  This work was supported by the National Natural  Science Foundation of China  Grant 12171105, the    Fok Ying Tong Education Foundation Grant 171094, and the  special foundation for Guangxi Ba Gui Scholars. All data generated or analysed during this study are included in this published article.}

\bibliographystyle{plain}
\bibliography{../tensor,../robust,../sparse_pca}

\begin{thebibliography}{10}

\bibitem{allen2012sparse}
G.~I. Allen.
\newblock {Sparse higher-order principal components analysis}.
\newblock In {\em International Conference on Machine Learning}, pages 27--36,
  April 2012.

\bibitem{bertsimas2022solving}
D.~Bertsimas, R.~Cory-Wright, and J.~Pauphilet.
\newblock Solving large-scale sparse {PCA} to certifiable (near) optimality.
\newblock {\em J. Mach. Learn. Res.}, 23(13):1--35, 2022.

\bibitem{chan2016approximability}
S.~O. Chan, D.~Papailliopoulos, and A.~Rubinstein.
\newblock On the approximability of sparse {PCA}.
\newblock In {\em Conference on Learning Theory}, pages 623--646, 2016.

\bibitem{aspremont2007a}
A.~d'Aspremont, L.~El~Ghaoui, M.~I. Jordan, and G.~R.~G. Lanckriet.
\newblock {A direct formulation for sparse {PCA} using semidefinite
  programming}.
\newblock {\em {SIAM} Rev.}, 49(3):434--448, January 2007.

\bibitem{de2000on}
L.~De~Lathauwer, B.~De~Moor, and J.~Vandewalle.
\newblock On the best rank-$1$ and rank-(${R}_1,{R}_2,\ldots,{R}_n$)
  approximation of higer-order tensors.
\newblock {\em SIAM J. Matrix Anal. Appl.}, 21:1324--1342, 2000.

\bibitem{d2014approximation}
A.~d’Aspremont, F.~Bach, and L.~El~Ghaoui.
\newblock Approximation bounds for sparse principal component analysis.
\newblock {\em Math. Program.}, 148(1-2):89--110, 2014.

\bibitem{he2012probability}
S.~He, B.~Jiang, Z.~Li, and S.~Zhang.
\newblock Probability bounds for polynomial functions in random variables.
\newblock {\em Math. Oper. Res.}, 39(3):889--907, 2014.

\bibitem{he2010approximation}
S.~He, Z.~Li, and S.~Zhang.
\newblock Approximation algorithms for homogeneous polynomial optimization with
  quadratic constraints.
\newblock {\em Math. Program., Ser. B}, 125:353--383, 2010.

\bibitem{luss2013conditional}
R.~Luss and M.~Teboulle.
\newblock Conditional gradient algorithms for rank-one matrix approximations
  with a sparsity constraint.
\newblock {\em SIAM Rev.}, 55(1):65--98, 2013.

\bibitem{madrid2017tensor}
O.~H. Madrid-Padilla and J.~Scott.
\newblock Tensor decomposition with generalized lasso penalties.
\newblock {\em J. Comput. Graph. Stat.}, 26(3):537--546, 2017.

\bibitem{mao2022several}
X.~Mao and Y.~Yang.
\newblock Several approximation algorithms for sparse best rank-1 approximation
  to higher-order tensors.
\newblock {\em J. Glob. Optim.}, 2022.

\bibitem{papalexakis2012k}
E.~E. Papalexakis, N.~D. Sidiropoulos, and R.~Bro.
\newblock From {K}-means to higher-way co-clustering: Multilinear decomposition
  with sparse latent factors.
\newblock {\em IEEE Trans. Signal Process.}, 61(2):493--506, 2012.

\bibitem{qi2018tensor}
L.~Qi, H.~Chen, and Y.~Chen.
\newblock {\em Tensor eigenvalues and their applications}, volume~39.
\newblock Springer, 2018.

\bibitem{sun2017store}
W.~W. Sun and L.~Li.
\newblock {STORE}: sparse tensor response regression and neuroimaging analysis.
\newblock {\em J. Mach. Learn. Res.}, 18(1):4908--4944, 2017.

\bibitem{sun2017provable}
W.~W. Sun, J.~Lu, H.~Liu, and G.~Cheng.
\newblock {Provable sparse tensor decomposition}.
\newblock {\em J. R. Stat. Soc. Ser. B-Stat. Methodol.}, 79(3):899--916, 2017.

\bibitem{wang2020sparse}
Y.~Wang, M.~Dong, and Y.~Xu.
\newblock A sparse rank-1 approximation algorithm for high-order tensors.
\newblock {\em Appl. Math. Lett.}, 102:106140, 2020.

\bibitem{witten2009a}
D.~M. Witten, R.~Tibshirani, and T.~Hastie.
\newblock {A penalized matrix decomposition, with applications to sparse
  principal components and canonical correlation analysis}.
\newblock {\em Biostatistics}, 10(3):515--534, 2009.

\bibitem{zhang2012cubic}
X.~Zhang, L.~Qi, and Y.~Ye.
\newblock The cubic spherical optimization problems.
\newblock {\em Math. Comput.}, 81(279):1513--1525, 2012.

\end{thebibliography}

\end{document}